\documentclass[journal]{IEEEtran}
\IEEEoverridecommandlockouts
\usepackage{amsmath,amsfonts,amssymb,euscript, graphicx, 
epsfig,
enumerate,float,afterpage, subfigure, ifthen}%
\usepackage{url}

\newtheorem{thm}{Theorem}

\newtheorem{lem}{Lemma}

\newcommand{\expect}[1]{\mathbb{E}\left\{#1\right\}}
\newcommand{\defequiv}{\mbox{\raisebox{-.3ex}{$\overset{\vartriangle}{=}$}}}

\newcommand{\bv}[1]{{\boldsymbol{#1} }}

\begin{document}

\title
  {Efficient Algorithms for Renewable Energy Allocation to Delay Tolerant Consumers}
\author{Michael J. Neely, Arash Saber Tehrani, Alexandros G. Dimakis%
$\vspace{-.3in}$
\thanks{The authors are with the Electrical Engineering department at the University
of Southern California, Los Angeles, CA.} 
\thanks{This material is supported in part  by one or more of 
the following: the DARPA IT-MANET program
grant W911NF-07-0028, 
the NSF Career grant CCF-0747525.}}

\markboth{}{Neely}

\maketitle

\begin{abstract}   
We investigate the problem of allocating energy from renewable sources
to flexible consumers in electricity markets.  We assume there is a renewable energy supplier that 
provides energy according to a time-varying (and possibly unpredictable) 
supply process.  The plant must serve consumers within a 
specified delay window, and incurs a cost of drawing energy from other (possibly non-renewable) 
sources if its own supply is not sufficient to meet the deadlines.   We formulate two stochastic optimization
problems: The first seeks to minimize the time average 
cost of using the other sources (and hence strives for the most efficient 
utilization of the renewable source).  The second  allows the 
renewable source to dynamically set
a price for its service, and seeks to maximize the resulting 
time average profit.  These problems are solved via the Lyapunov
optimization technique.  
Our resulting algorithms do
not require knowledge of the statistics of the time-varying supply
and demand processes and are robust to arbitrary sample path variations.
\end{abstract} 

\section{Introduction} 

The highly variable and unpredictable nature of some renewable energy sources (such as wind and solar)
has been a major obstacle to their integration. 
For example, a recent study conducted by Enernex for wind power integration in
Minnesota~\cite{enerex} indicates that the variability and day-ahead
forecast errors will result in an additional $\$2.11-\$4.41$  (for
$15\%$ and $25\%$ penetration) per MWh of delivered wind power. Along
the same lines, the CAISO report~\cite{caiso} predicted that ten minute real-time
energy prices could increase substantially due to wind forecasting
errors and identified day-ahead and same-day forecasts and modeling as
important tasks for integration of renewable resources.

The necessity to offset variability by stand-by generators and system backup investments substantially
increases the cost of renewables. One approach that can mitigate this problem is to couple this supply variability to demand side flexibility~\cite{pap1,pap2,pap3}. The renewable power suppliers could sell their energy at a lower price to consumers that are willing to wait in a queue, given that it will be served to them within a pre-agreed deadline. This essentially allows a lower price of renewable energy to consumers willing to provide this extra time flexibility.  The renewable power supplier can now use this 
flexibility to deliver the energy when it is available.\footnote{Note that in this paper we assume no energy storage although it can be naturally incorporated into our framework.}
The supplier will sometimes, hopefully rarely, be in a situation when a prior deadline commitment cannot be matched and will have to purchase the extra energy from the energy spot market (or maintain a costly system backup). Papavasiliou and Oren~\cite{pap3} introduced this problem and proposed an exact backward dynamic programming algorithm and an efficient approximate dynamic programming algorithm for the scheduling decisions of the renewable energy supplier.

In this paper we build a similar model and utilize the technique of Lyapunov optimization initially developed in \cite{now}\cite{neely-energy-it}\cite{neely-thesis} for dynamic control of queueing
systems for wireless networks. We show that the queuing model naturally fits in the renewable supplier
scheduling problem and present a simple energy allocation algorithm that does not require prior statistical information and 
is provably close to optimal. The proposed framework can be extended to include pricing, multiple queues (with different deadlines) and different objective functions, building on the general results from~\cite{now}. We finally evaluate the proposed algorithm on actual CAISO spot market and wind energy production data and show substantial reduction to the operating costs for the renewable supplier 
compared to a simple greedy algorithm. 

In particular, we consider a single renewable energy plant that operates in discrete time with unit
timeslots $t \in \{0, 1, 2, \ldots\}$, and provides $s(t)$ units of energy on each slot $t$. 
The $s(t)$ process corresponds to the renewable supply and is assumed to be time varying and unpredictable. 
Since we assume no storage, the energy $s(t)$ must either be used or
wasted.  Demands for this energy arrive randomly according to a process $a(t)$ (being the
amount of energy that is requested on slot $t$).  We assume that consumers requesting energy are flexible, and can 
tolerate their energy requests being satisfied with some delay.   The requests are thus stored in a queue.  
Every slot $t$, we use all of our supply $s(t)$ to serve the requests in the queue in a First-In-First-Out (FIFO) manner.  
However, this may not be enough to 
meet all of the requests within a timely manner, and hence we also decide to purchase an amount of energy $x(t)$ from an 
outside (possibly non-renewable) plant.  Letting $Q(t)$ represent the total energy requests in our queue on slot $t$, we
have the following update equation: 
\begin{equation} \label{eq:q-update} 
 Q(t+1) = \max[Q(t) - s(t) - x(t), 0] + a(t) 
\end{equation} 

The value $x(t)$ is a control decision variable, and incurs a cost $x(t)\gamma(t)$ on slot $t$, where $\gamma(t)$ is a process
that specifies the per-unit-cost of using the outside energy supply on slot $t$. 
The value of $\gamma(t)$ can represent a current market price for guaranteed 
energy services from (possibly non-renewable) sources.   As such, the decision to use $x(t)$ units of energy on slot $t$ 
means the outside source agrees to provide this much energy at time
$t+K$ for some fixed (and small) integer $K\geq 0$, for the price $x(t)\gamma(t)$. Without loss of generality, we 
assume throughout
that $K=0$, so that the energy request is removed from our queue on the same slot in which we decide to 
use  the outside source. In the actual  implementation, 
requests that are served from the outside source can be removed from the primary queue
$Q(t)$ but must still wait an additional $K$ slots.

We first look at the problem of choosing $x(t)$ to stabilize our queue $Q(t)$ while minimizing 
the time average of the cost $x(t)\gamma(t)$ and also providing a guarantee on the maximum delay $D_{max}$ 
spent in the queue.  If the future values of 
supply, demand, and market price values $(s(t), a(t), \gamma(t))$ were known in advance, one could in principle
make $x(t)$ decisions
that minimize total time average cost, possibly choosing $x(t)=0$ for all $t$ if it is possible to meet all demands
using only the renewable energy $s(t)$.   The challenge is to provide an efficient algorithm without knowing the future. 
To this end, we first assume the vector process $(s(t), a(t), \gamma(t))$ is i.i.d. over slots but has an
unknown probability distribution.  Under this assumption, we 
develop an algorithm, parameterized by a positive value $V$, 
 that comes within $O(1/V)$ of the minimum time
average cost required to stabilize the queue, with a worst-case delay guarantee that is $O(V)$. 
The parameter $V$ can be tuned as desired to provide average cost arbitrarily close to optimal, with a tradeoff
in delay. 
We further show that the same algorithm is provably 
robust to non-i.i.d. situations, and operates efficiently even
for arbitrary sample paths for $(s(t), a(t), \gamma(t))$.   Finally, we extend the problem to consider pricing decisions
at the renewable energy source, so that the requests $a(t)$ are now influenced by the current prices.  In this case, 
we design a related algorithm that maximizes time average profit. 

The Lyapunov optimization technique we use~\cite{now}\cite{neely-energy-it}\cite{neely-thesis}
is related to the primal-dual and 
fluid-model techniques in \cite{atilla-fairness}\cite{stolyar-greedy}\cite{vijay-allerton02}\cite{prop-fair-down}. 
The work in \cite{now}\cite{neely-energy-it}\cite{neely-thesis} establishes a general $[O(1/V), O(V)]$ 
performance-congestion
tradeoff for stochastic network optimization problems with i.i.d. (and more general ergodic) processes.  
Recent work in \cite{neely-universal-scheduling}\cite{neely-inventory-control-arxiv} provides similar results on a sample path basis, without any probabilistic assumptions.  We apply these results in our current paper.  Further, we 
extend the theory by introducing a novel virtual queue that turns an average delay constraint of $O(V)$ (which is achievable
with the prior analytical techniques) into a \emph{worst case delay guarantee} that is also $O(V)$. 

It is useful to distinguish the proposed Lyapunov optimization method that we use in this paper from dynamic programming techniques.  
Dynamic programming can be used to solve stronger versions of our problem (such as minimizing average cost subject to a delay constraint) see e.g.~\cite{pap3}.  However, dynamic programming requires more stringent 
system modeling assumptions, has a more complex
solution that typically requires knowledge of the supply, demand, and market price probabilities, and 
cannot necessarily adapt if these probabilities change and/or if there are unmodeled correlations
in the actual processes.    It involves
computation of a value function that can be difficult when the state space of the system is large, and suffers
from a \emph{curse of dimensionality} when applied to large dimensional systems (such as systems with many queues).

In contrast, Lyapunov optimization is relatively simple to implement, does not need a-priori statistical knowledge, and 
is robust to non-i.i.d. and non-ergodic behavior. 
Further, it has no curse of dimensionality and hence can be applied just as easily in extended formulations that have multiple queues corresponding to 
multiple customers requesting different deadlines, contrary to dynamic programming~\cite{pap3} which would require exponential complexity in the number of users. 

The reason for this efficiency is that Lyapunov optimization relaxes the question that dynamic programming asks:  Rather than minimizing time average cost subject to a delay constraint, it seeks to push time average cost towards the more ambitious
minimum over all possible algorithms that can stabilize the queue (without regard to the delay constraint).  It then specifies
an explicit bound on the resulting queue congestion, which depends on the desired 
proximity to the minimum cost (as defined by the $[O(1/V), O(V)]$ performance-congestion tradeoff).  
However, the resulting time average queue congestion (and delay) that is achieved is not necessarily the optimal that could be
achieved over all possible algorithms that yield the same time average performance cost.

In the next section, we formulate the basic model under the assumption that the
$(s(t), a(t), \gamma(t))$ vector is i.i.d. over slots, and present the main allocation algorithm. 
Section \ref{section:pricing} 
extends to the case when the renewable power source can set a price for its services.   These algorithms are provably
robust to non-i.i.d. situations and arbitrary sample paths of events, as shown in Section \ref{section:non-iid}.
Section \ref{section:exp} presents an experimental evaluation of our algorithm on a real six-month data set and shows substantial gains over a simple greedy scheduling algorithm.

\section{The Dynamic Allocation Algorithm}  \label{section:model} 

Suppose that the supply process $s(t)$, the request process $a(t)$, and the market price process $\gamma(t)$, as 
described in the introduction, form a vector $(s(t), a(t), \gamma(t))$ that is i.i.d. over slots with some unknown probability
distribution.  We further assume the values of $s(t)$, $a(t)$, $\gamma(t)$ are deterministically bounded by finite constants
$s_{max}$, $a_{max}$, $\gamma_{max}$, so that: 
\begin{eqnarray}  
0 \leq s(t) \leq s_{max}  \: , 
0 \leq a(t) \leq a_{max}   \:  , 
0 \leq \gamma(t) \leq \gamma_{max} \:  \forall t  \label{eq:sample-path-bounds} 
\end{eqnarray}

The queue backlog $Q(t)$ evolves according to (\ref{eq:q-update}). The decision variable $x(t)$ is chosen every slot $t$
in reaction to the current $(s(t), a(t), \gamma(t))$ (and possibly additional queue state information) subject to
the constraint $0 \leq x(t) \leq x_{max}$ for all $t$, 
where $x_{max}$ is a finite upper bound.  We assume that $x_{max} \geq a_{max}$ so that it is always possible to 
stabilize the queue $Q(t)$ (and this can be done with one slot delay if we choose $x(t) = x_{max}$ for all $t$). 
Define $\overline{c}$ as the time average cost incurred by our control policy (assuming temporarily that our policy 
yields such a well defined limit): 
\begin{eqnarray*}
&\overline{c} \defequiv \lim_{t\rightarrow\infty} \frac{1}{t}\sum_{\tau=0}^{t-1} \expect{\gamma(t)x(t)}&
\end{eqnarray*}
We want to find an allocation algorithm that chooses $x(t)$ over time to solve: 
\begin{eqnarray}
\mbox{Minimize:} && \overline{c} \label{eq:problem1}  \\
\mbox{Subject to:} &1)&  \overline{Q} < \infty \label{eq:problem2} \\
&2)&  0 \leq x(t) \leq x_{max} \: \: \forall t \label{eq:problem3} 
\end{eqnarray}
where $\overline{Q}$ is the time average expected queue backlog, defined: 
\begin{eqnarray*}
&\overline{Q} \defequiv \limsup_{t\rightarrow\infty} \frac{1}{t}\sum_{\tau=0}^{t-1} \expect{Q(\tau)}&
\end{eqnarray*}
Define $c^*$ as the infimum time average cost associated with the above problem, considering all possible
ways of choosing $x(t)$ over time.   The value of $c^*$ is an ambitious target 
because the above problem is defined
only in terms of a queue stability constraint and does not impose any additional delay constraint.  We shall construct
a solution, parameterized by a constant $V>0$,  that satisfies the constraints of the above problem and pushes
the average cost within $O(1/V)$ of the optimal value 
$c^*$.  Further, we show that our algorithm has the 
additional property that worst case delay is no more than
$O(V)$.

\subsection{The Delay-Aware Virtual Queue} 

We solve the above problem while also maintaining finite worst case delay using the following novel
``virtual queue'' $Z(t)$:  Fix a parameter $\epsilon>0$, to be specified later.   
Define $Z(0)=0$, and define the virtual queue
$Z(t)$ for $t \in \{0, 1, 2, \ldots\}$ according to the following update: 
\begin{eqnarray} 
Z(t+1) = \max[Z(t) - s(t) - x(t) + \epsilon1_{\{Q(t)>0\}}, 0] \label{eq:z-update}
\end{eqnarray}  
where $1_{\{Q(t)>0\}}$ is an indicator function that is 1 if $Q(t)>0$, and zero else. 
The intuition is that $Z(t)$ has the same service process as $Q(t)$ (being $s(t) + x(t)$), but now has an arrival
process that adds $\epsilon$ whenever the actual queue backlog is non-empty.  This ensures that $Z(t)$ grows
if there are requests in the $Q(t)$ queue that have not been serviced for a long time. 
 If we can control the system to ensure that the queues $Q(t)$ and $Z(t)$ have finite upper bounds, then 
 we can ensure all requests are served with a worst case delay given in the 
following lemma.\footnote{In the case when requests are served by the outside source with an additional delay
$K>0$, then this bound is modified in the actual implementation to $\lceil(Q_{max} + Z_{max})/\epsilon\rceil + K$.} 

\begin{lem} \label{lem:z} (Worst Case Delay) 
Suppose the system is controlled so that $Z(t) \leq Z_{max}$ and $Q(t) \leq Q_{max}$ for all $t$, 
for some positive constants $Z_{max}$ and $Q_{max}$. Then all requests are fulfilled with a maximum delay of $D_{max}$
slots, where: 
\begin{equation}
D_{max} \defequiv \lceil(Q_{max} + Z_{max})/\epsilon\rceil  \label{eq:Dmax} 
\end{equation} 
\end{lem} 
\begin{proof} 
Consider any slot $t$ for which $a(t)>0$.  We show that the requests $a(t)$ are fulfilled on or before
time $t + D_{max}$.  Suppose not (we shall reach a contradiction).  Then during 
slots $\tau \in \{t+1, \ldots, t+D_{max}\}$ it must be that $Q(\tau) > 0$ (else the requests $a(t)$ would have
been served before slot $\tau$).  Thus, $1_{\{Q(\tau)>0\}} = 1$, and
from (\ref{eq:z-update}) we have that for all $\tau \in \{t+1, \ldots, t+D_{max}\}$: 
\[ Z(\tau+1) \geq Z(\tau) - s(\tau) - x(\tau) + \epsilon \]
Summing the above over $\tau \in \{t+1, \ldots, t+D_{max}\}$ yields: 
\[ Z(t+D_{max}+1) - Z(t+1) \geq - \sum_{\tau=t+1}^{t+D_{max}} [s(\tau) + x(\tau)] + D_{max}\epsilon \]
Rearranging and using the fact that $Z(t+1) \geq 0$ and $Z(t+D_{max}+1) \leq Z_{max}$ yields: 
\begin{eqnarray} 
 &\sum_{\tau=t+1}^{t+D_{max}} [s(\tau) + x(\tau)] \geq D_{max}\epsilon - Z_{max}& \label{eq:lemz1} 
 \end{eqnarray} 
 Now note that the requests $a(t)$ are first available for service at time $t+1$, and are part of the 
 backlog $Q(t+1)$ (see (\ref{eq:q-update})).  Because $Q(t+1) \leq Q_{max}$ and because service is FIFO, 
 these requests $a(t)$
 are served on or before time $t+D_{max}$ whenever 
 there are at least $Q_{max}$ units of energy served
 during the interval $\tau \in \{t+1, \ldots, t+D_{max}\}$. Because we have assumed the requests $a(t)$ are \emph{not}
 served by time $t+D_{max}$, it must be that
 $\sum_{\tau=t+1}^{t+D_{max}} [s(\tau) + x(\tau)] < Q_{max}$. 
Using this in (\ref{eq:lemz1}) yields: 
\[ Q_{max} > D_{max} \epsilon - Z_{max} \]
This implies that $D_{max} < (Q_{max} + Z_{max})/\epsilon$, 
contradicting the definition of $D_{max}$ in (\ref{eq:Dmax}). 
\end{proof} 

\subsection{Lyapunov Optimization} 

Define $\bv{\Theta}(t) \defequiv (Z(t), Q(t))$ as the concatenated vector of the real and virtual queues. 
As a scalar measure of the congestion in both the $Z(t)$ and $Q(t)$ queues, we define the following
Lyapunov function: $L(\bv{\Theta}(t)) \defequiv \frac{1}{2}[Z(t)^2 + Q(t)^2]$. 
Define the conditional 1-slot Lyapunov drift as follows: 
\begin{equation} \label{eq:delta-def} 
\Delta(\bv{\Theta}(t)) \defequiv \expect{L(\bv{\Theta}(t+1)) - L(\bv{\Theta}(t))|\bv{\Theta}(t)} 
\end{equation} 
Following the drift-plus-penalty framework of \cite{now}\cite{neely-energy-it}\cite{neely-thesis}, our control
algorithm is designed to observe the current queue states $Z(t)$, $Q(t)$ and the current $(s(t), a(t), \gamma(t))$ vector, and 
to make a decision $x(t)$ (where $0 \leq x(t) \leq x_{max}$) to 
minimize a bound on the following expression every slot $t$: 
\[ \Delta(\bv{\Theta}(t)) + V\expect{\gamma(t)x(t)|\bv{\Theta}(t)} \]
where $V$ is a positive 
parameter that will be useful to affect a performance-delay tradeoff. 
We first compute a bound on the above drift-plus-penalty expression. 

\begin{lem} \label{lem:drift-bound} (Drift Bound) 
For any control policy that satisfies $0 \leq x(t) \leq x_{max}$ for all $t$,  
the drift-plus-penalty expression for all slots $t$ satisfies: 
\begin{eqnarray} 
&& \hspace{-.3in}\Delta(\bv{\Theta}(t)) + V\expect{\gamma(t)x(t)|\bv{\Theta}(t)} \leq  B +V\expect{\gamma(t)x(t)|\bv{\Theta}(t)} \nonumber \\
&& \hspace{+1in} + Q(t)\expect{a(t) - s(t) - x(t)|\bv{\Theta}(t)} \nonumber \\
&& \hspace{+1in} + Z(t)\expect{\epsilon - s(t) - x(t)|\bv{\Theta}(t)}  \label{eq:drift-bound} 
\end{eqnarray} 
where the constant $B$ is defined: 
\begin{equation} \label{eq:B} 
B \defequiv \frac{(s_{max}+x_{max})^2 + a_{max}^2}{2} + \frac{\max[\epsilon^2, (s_{max} + x_{max})^2]}{2} 
\end{equation} 
\end{lem} 
\begin{proof} 
See Appendix A. 
\end{proof} 

\subsection{The Dynamic Algorithm} 

Minimizing the right-hand-side of the drift-plus-penalty bound (\ref{eq:drift-bound}) every slot $t$ leads to the 
following dynamic algorithm:  Every slot $t$, observe $Z(t)$, $Q(t)$, $(s(t), a(t), \gamma(t))$, and choose $x(t)$
according to the following optimization: 
\begin{eqnarray*} 
\mbox{Minimize:} & x(t)[V\gamma(t) - Q(t)- Z(t)]\\
\mbox{Subject to:} & 0 \leq x(t) \leq x_{max} 
\end{eqnarray*}
Then update the actual and virtual queues $Q(t)$ and $Z(t)$ by (\ref{eq:q-update}) and (\ref{eq:z-update}). 
The above minimization for the 
$x(t)$ decision reduces to the following simple threshold rule: 
\begin{equation} \label{eq:x-choice} 
x(t) = \left\{ \begin{array}{ll}
                          0  &\mbox{ if $Q(t) + Z(t) \leq V\gamma(t)$} \\
                             x_{max}  & \mbox{ otherwise} 
                            \end{array}
                                 \right. 
\end{equation} 

The above $x(t)$ value drives the queueing updates (\ref{eq:q-update}) and (\ref{eq:z-update}).  However, 
note by the $\max[\cdot,0]$ structure of the $Q(t)$ update in (\ref{eq:q-update}) that we may not need to purchase
the full $x(t)$ units of energy from the outside plant on slot $t$.  Indeed, define $\tilde{x}(t)$ as the \emph{actual} 
amount purchased from the plant, given by: 
\begin{equation} \label{eq:x-tilde} 
\tilde{x}(t) \defequiv  \left\{ \begin{array}{ll}
                          x(t)&\mbox{ if $Q(t) - s(t) \geq  x(t)$} \\
                             \min[Q(t) - s(t), 0]  & \mbox{ otherwise} 
                            \end{array}
                                 \right.
\end{equation} 
Then we have $\tilde{x}(t) \leq x(t)$ for all $t$.

\begin{thm} \label{thm:performance} (Performance Analysis) Suppose 
$x_{max} \geq \max[a_{max}, \epsilon]$.  
If $Q(0) = Z(0) = 0$, and if the above dynamic algorithm is implemented with any fixed $\epsilon\geq0$ and 
$V>0$ for all $t \in \{0, 1, 2, \ldots\}$, then: 

a) The queues $Q(t)$ and $Z(t)$ are deterministically bounded by $Q_{max}$ and $Z_{max}$ every slot $t$, where: 
\begin{eqnarray}
Q_{max} \defequiv  V\gamma_{max} + a_{max}  \: \: , \: \: 
Z_{max} \defequiv V\gamma_{max} + \epsilon \label{eq:Qmax} 
\end{eqnarray}

b) The worst case delay of any request is: 
\begin{equation} \label{eq:Dmax2} 
D_{max} = \left\lceil(2V\gamma_{max} + a_{max} + \epsilon)/\epsilon \right\rceil 
\end{equation} 

c) If the vector $(s(t), a(t), \gamma(t))$ is i.i.d. over slots, and if 
the $\epsilon$ parameter is chosen to satisfy $\epsilon \leq \max[\expect{a(t)}, \expect{s(t)}]$, 
then for all slots $t>0$ the time average cost satisfies: 
\begin{eqnarray*} 
 \frac{1}{t}\sum_{\tau=0}^{t-1} \expect{\gamma(\tau)\tilde{x}(\tau)} \leq 
\frac{1}{t}\sum_{\tau=0}^{t-1} \expect{\gamma(\tau)x(\tau)}  
 \leq c^* + B/V 
 \end{eqnarray*}
where $B$ is defined in (\ref{eq:B}). 
\end{thm}

The above theorem demonstrates the $[O(1/V), O(V)]$ cost-delay tradeoff, where time average cost
is within $B/V$ of the minimum possible time average cost $c^*$ required for queue stability, and worst
case delay is proportional to $V/\epsilon$.  To obtain the smallest $D_{max}$, the 
$\epsilon$ value should be chosen as large
as possible while still maintaining $\epsilon \leq \max[\expect{a(t)}, \expect{s(t)}]$.  
We can choose $\epsilon = \expect{a(t)}$ if this
expectation is known.  Using $\epsilon =0$ preserves parts (a) and (c) but does not give a finite $D_{max}$. 
More discussion of the $\epsilon =0$ case is given in Section \ref{section:exp}. 

\subsection{Proof of Theorem \ref{thm:performance}}

\begin{proof} (Theorem \ref{thm:performance} part (a)) 
We first show that $Q(t) \leq V\gamma_{max} + a_{max}$ for all $t$. This is clearly true for $t=0$ (because $Q(0)=0$). 
Suppose it holds for slot $t$.  We show it also holds for slot $t+1$.  
Consider the case when $Q(t) \leq V\gamma_{max}$.  Then $Q(t+1) \leq V\gamma_{max} + a_{max}$,  because 
the queue can increase by at most $a_{max}$ on any slot (see dynamics (\ref{eq:q-update})). Thus, the result
holds in this case. 

Now consider the opposite case when  $V\gamma_{max} < Q(t) \leq V\gamma_{max} + a_{max}$.  In this case, 
we have: 
\[ Q(t) + Z(t) \geq Q(t) > V\gamma_{max} \geq V\gamma(t) \]
and hence the algorithm will choose $x(t) = x_{max}$ according to (\ref{eq:x-choice}).  If $Q(t) - x_{max} - s(t) > 0$, then
on slot $t$ we serve at least $x_{max}$ units of data.  Because arrivals $a(t)$ are at most $a_{max}$ (and $a_{max} \leq x_{max}$), the queue cannot increase on the next slot and so $Q(t+1) \leq Q(t) \leq V\gamma_{max} + a_{max}$. 
Finally, if $Q(t) - x_{max} - s(t) \leq 0$, then by (\ref{eq:q-update}) we have
$Q(t+1) = a(t) \leq a_{max}$, again being less than or equal to 
$V\gamma_{max} + a_{max}$.

Therefore, $Q(t) \leq V\gamma_{max} + a_{max}$ for all $t$. 
The proof that $Z(t) \leq V\gamma_{max} + \epsilon$ for all $t$ is similar and  omitted for brevity. 
\end{proof} 

\begin{proof} (Theorem \ref{thm:performance} part (b)) 
This follows immediately from Lemma \ref{lem:z} together with part (a). 
\end{proof}

The proof of Theorem \ref{thm:performance} part (c) requires a preliminary lemma
from  \cite{neely-energy-it}. 
To introduce the lemma, define a \emph{$(s,a,\gamma)$-only policy} to be one  
 that observes the current vector $(s(t), a(t), \gamma(t))$
and makes a stationary and randomized decision $x^*(t)$ based purely on this vector (and independent of the queue
backlogs or past system history), subject to the constraint $0 \leq x^*(t) \leq x_{max}$.   

\begin{lem} \label{lem:optimality} (Characterizing Optimality \cite{neely-energy-it}) If the vector
$(s(t), a(t), \gamma(t))$ is i.i.d. over
slots, then there 
exists a $(s,a,\gamma)$-only policy $x^*(t)$ that satisfies: 
\begin{eqnarray} 
\expect{\gamma(t)x^*(t)} &=& c^* \label{eq:x1opt} \\
\expect{s(t) + x^*(t)} &\geq& \expect{a(t)}  \label{eq:x2opt} 
\end{eqnarray} 
where $c^*$ is the infimum time average cost in the stochastic optimization problem (\ref{eq:problem1})-(\ref{eq:problem3}), and 
the above expectations are with respect to the stationary distribution of the vector
 $(s(t), a(t), \gamma(t))$ and the possibly randomized action $x^*(t)$ made in reaction to this vector.
\end{lem}
\begin{proof} (Lemma \ref{lem:optimality})
This follows as a special case of results in \cite{neely-energy-it}. 
\end{proof} 

\begin{proof} (Theorem \ref{thm:performance} part (c)) 
We have assumed that $\epsilon \leq \max[\expect{a(t)}, \expect{s(t)}]$.  We first prove the result
for the case when $\epsilon \leq \expect{a(t)}$.  
On every slot $t$, 
the dynamic choice of $x(t)$ in (\ref{eq:x-choice}) minimizes the right-hand-side of the drift bound (\ref{eq:drift-bound}) 
(given the observed queue sizes $\bv{\Theta}(t) \defequiv (Q(t), Z(t))$), over all alternative choices $x^*(t)$ 
that satisfy the required bounds $0 \leq x^*(t) \leq x_{max}$ (including
randomized choices for $x^*(t)$). Thus, by (\ref{eq:drift-bound}) we have: 
\begin{eqnarray*}
&& \hspace{-.3in} \Delta(\bv{\Theta}(t)) + V\expect{\gamma(t)x(t)|\bv{\Theta}(t)} \leq B  + V\expect{\gamma(t)x^*(t)|\bv{\Theta}(t)} \\
&&\hspace{+.8in} + Q(t)\expect{a(t) - s(t) - x^*(t)|\bv{\Theta}(t)} \\
&&\hspace{+.8in} + Z(t)\expect{\epsilon - s(t) - x^*(t)|\bv{\Theta}(t)} 
\end{eqnarray*} 
where $x^*(t)$ is any alternative (possibly randomized) decision. 
Plugging the $(s, a, \gamma)$-only policy $x^*(t)$ from (\ref{eq:x1opt})-(\ref{eq:x2opt}) (known to exist by Lemma \ref{lem:optimality}) into the right hand side of the above inequality for slot $t$, and noting that this policy makes
decisions independent of queue backlogs, yields: 
\begin{eqnarray}
&& \hspace{-.8in} \Delta(\bv{\Theta}(t)) + V\expect{\gamma(t)x(t)|\bv{\Theta}(t)} \leq B  + Vc^* \label{eq:drift-almost-done} 
\end{eqnarray} 
where we have used the fact that: 
\begin{eqnarray} 
&&\hspace{-.3in} \expect{a(t) - s(t) - x^*(t) | \bv{\Theta}(t)} \nonumber \\
&&= \expect{a(t) - s(t) - x^*(t)} \leq 0 \label{eq:proofc1} \\
&&\hspace{-.3in}  \expect{\epsilon - s(t) - x^*(t)|\bv{\Theta}(t)} \nonumber \\
&& = \expect{\epsilon - s(t) - x^*(t)} \leq 0 \label{eq:proofc2} 
\end{eqnarray} 
where (\ref{eq:proofc1}) follows from (\ref{eq:x2opt}) and the fact that the $(s, a, \gamma)$-only policy $x^*(t)$ 
is i.i.d. over slots and hence independent 
of queue backlogs $\bv{\Theta}(t)$, and (\ref{eq:proofc2}) follows from (\ref{eq:x2opt}) together
with  the fact that $\expect{a(t)} \geq \epsilon$.

Taking expectations of (\ref{eq:drift-almost-done}) 
and using the law of iterated expectations with the definition of $\Delta(\bv{\Theta}(t))$ in (\ref{eq:delta-def}) yields: 
\[ \expect{L(\bv{\Theta}(t+1))} - \expect{L(\bv{\Theta}(t))} + V\expect{\gamma(t)x(t)}  \leq B + Vc^* \]
The above holds for all slots $t>0$.  Summing over $t \in \{0, 1, \ldots, M-1\}$ for some positive integer $M$ yields: 
\begin{eqnarray*}
 \expect{L(\bv{\Theta}(M))} - \expect{L(\bv{\Theta}(0))} + \sum_{t=0}^{M-1}V\expect{\gamma(t)x(t)}  \leq \\
 BM + VMc^*
 \end{eqnarray*}
Using the fact that $L(\bv{\Theta}(0)) = 0$ (because all queues are initially empty), and that $L(\bv{\Theta}(M)) \geq 0$
(because the Lyapunov function is non-negative) and dividing by $VM$ yields: 
\begin{eqnarray*}
&\frac{1}{M}\sum_{t=0}^{M-1} \expect{\gamma(t)x(t)} \leq c^* + B/V& 
\end{eqnarray*}
This holds for all $M>0$, proving the result for the case when $\epsilon \leq \expect{a(t)}$. 

We have only used the assumption that $\epsilon \leq \expect{a(t)}$ to ensure the inequality 
(\ref{eq:proofc2}) holds.  If $\epsilon \leq \expect{s(t)}$, then clearly (\ref{eq:proofc2}) 
holds, regardless of the value of $\expect{a(t)}$. 
Thus, the result holds whenever $\epsilon \leq \max[\expect{a(t)}, \expect{s(t)}]$, proving the theorem. 
\end{proof}

\section{Pricing for Maximum Profit} \label{section:pricing} 

We now extend the problem to consider  pricing decisions.   Instead of a process $a(t)$ that represents requests arriving
at slot $t$, we define a process $y(t)$, called the \emph{demand state} on slot $t$.  The demand state captures any
properties of the demand that may affect requests for the renewable energy source in reaction to the price advertised
on slot $t$.  A simple example is when $y(t)$ can take one of two possible values, such as HIGH and LOW, representing
different demand conditions (such as during peak times or non-peak times for requesting energy).  Another
example is when $y(t)$ represents the number of consumers willing to purchase renewable energy on slot $t$.  
We assume
the demand state $y(t)$ is known at the beginning of each slot $t$ (we show a particular case where $y(t)$
does not need to be known after our algorithm is stated). 

Every slot $t$, in addition to choosing the amount of energy $x(t)$ purchased from outside sources, the renewable
energy plant makes a binary decision $b(t) \in \{0, 1\}$, where $b(t) = 1$ represents 
a willingness to accept new requests on slot $t$, and $b(t)=0$ means no requests will be accepted.  If $b(t) = 1$
is chosen, the plant 
also chooses a per-unit-energy price $p(t)$ within an interval $0 \leq p(t) \leq p_{max}$, where $p_{max}$
is a pre-established maximum price.  The arriving requests $a(t)$ are then influenced by the current price $p(t)$, the
current market price $\gamma(t)$, and the current demand state $y(t)$, according to a general \emph{demand function} 
$F(p, y, \gamma)$.  Specifically, the values of $a(t)$ are assumed to be conditionally i.i.d. over all slots with the 
same $p(t)$, $y(t)$, $\gamma(t)$, and satisfy: 
\[ \expect{a(t) | p(t), y(t), \gamma(t), b(t)=1} = F(p(t), y(t), \gamma(t)) \] 
We assume the function $F(p,y,\gamma)$ is continuous in $p$ for each given $y$ and $\gamma$.\footnote{This continuity
is only used to ensure the  resulting min-drift decision has a well defined minimizing price $p(t)$ every slot.}
We further assume the arrivals $a(t)$ continue to be worst-case bounded by $a_{max}$, regardless of $p(t)$, $y(t)$, 
$\gamma(t)$.  The queue iteration $Q(t)$ still operates according to (\ref{eq:q-update}), with the understanding that 
 $a(t)$ is now influenced
by the pricing decisions.  Let $\phi(t)$ represent the instantaneous profit earned on slot $t$, defined as: 
\[ \phi(t) = b(t)p(t)a(t) - \gamma(t)x(t) \]
   We now consider the following problem: 
\begin{eqnarray}
\mbox{Maximize:} && \overline{\phi} \label{eq:newopt1}  \\
\mbox{Subject to:} &1)&  \overline{Q} < \infty \label{eq:newopt2} \\
&2)& 0 \leq x(t) \leq x_{max} \forall t \label{eq:newopt3} \\
&3)& b(t) \in \{0,1\} \:  ,  \: 0 \leq p(t) \leq p_{max} \forall t \label{eq:newopt4} 
\end{eqnarray}
where $\overline{\phi}$ is defined as the limiting time average profit: 
\begin{eqnarray*}
&\overline{\phi} \defequiv \lim_{t\rightarrow\infty} \frac{1}{t}\sum_{\tau=0}^{t-1} \expect{\phi(\tau)}&
\end{eqnarray*}

To solve the problem, we use the same queueing structure for $Q(t)$ in (\ref{eq:q-update}) and the same
virtual queue structure for $Z(t)$ in (\ref{eq:z-update}), and use the same Lyapunov function $L(\bv{\Theta}(t))$ as
defined before (recall that $\bv{\Theta}(t)$ is defined as the vector $(Q(t), Z(t))$).  However, we now consider
the ``penalty'' $-\phi(t)$, and so the drift-plus-penalty technique seeks to choose a vector that minimizes a bound on: 
\[ \Delta(\bv{\Theta}(t))  -V\expect{\phi(t)|\bv{\Theta}(t)} \]
Using the same analysis as Lemma \ref{lem:drift-bound}, we can show the following bound on this drift-plus-penalty
expression: 
\begin{eqnarray}
\Delta(\bv{\Theta}(t)) - V\expect{\phi(t)|\bv{\Theta}(t)} \leq B  \nonumber \\
- V\expect{b(t)p(t)F(p(t), y(t), \gamma(t)) - \gamma(t)x(t)|\bv{\Theta}(t)} \nonumber \\
+ Q(t)\expect{b(t)F(p(t), y(t), \gamma(t)) - s(t) - x(t)|\bv{\Theta}(t)} \nonumber \\
+ Z(t)\expect{\epsilon - s(t) - x(t)|\bv{\Theta}(t)} \label{eq:price-drift}  
\end{eqnarray}

Our joint energy-allocation and pricing algorithm observes the current system state on each slot $t$, and 
chooses $b(t)$, $p(t)$, and $x(t)$ to minimize the right-hand side of the above drift expression (given the observed
$\bv{\Theta}(t)$).  This reduces to the following: Every slot $t$, observe queues $Q(t)$, $Z(t)$, and 
observe $s(t)$, $\gamma(t)$, $y(t)$.  Then choose a price $p(t)$ and an allocation $x(t)$ as follows: 

\begin{itemize} 
\item (Pricing $p(t)$)  Choose $p(t)$ as the solution to: 
\begin{eqnarray*}
\mbox{Max:} & F(p(t), y(t), \gamma(t))(Vp(t)
 - Q(t))\\
\mbox{S.t.:} & 0 \leq p(t) \leq p_{max} 
\end{eqnarray*}
If the resulting maximum value is non-negative, choose $b(t)=1$.  Else choose $b(t)=0$ so that no
new requests are allowed on slot $t$.
\item (Allocating $x(t)$) Choose $x(t)$ according to 
(\ref{eq:x-choice}). 
\item (Queue Updates) Update $Q(t)$ and $Z(t)$ by (\ref{eq:q-update}) and (\ref{eq:z-update}). 
\end{itemize} 

This pricing pricing policy does not need to know the demand state $y(t)$ in the special case when
$F(p(t), y(t), \gamma(t)) = y(t)\hat{F}(p(t), \gamma(t))$, so that demand state simply scales
the demand function.  This pricing  structure is similar to that
considered in \cite{longbo-two-price-ton} for 
wireless service providers. 

\subsection{Defining Optimality}  \label{section:price-optimality} 

We define a \emph{$(s, y, \gamma)$-only policy} as one that jointly chooses $x^*(t)$, $b^*(t)$, $p^*(t)$
subject to $0 \leq x^*(t) \leq x_{max}$, $b^*(t)\in\{0,1\}$, $0 \leq p^*(t) \leq p_{max}$ according to a stationary and randomized
decision that depends only on $s(t)$, $y(t)$, $\gamma(t)$.  As in \cite{neely-energy-it}, it can be shown that
the supremum time average profit $\phi^*$ associated with the problem (\ref{eq:newopt1})-(\ref{eq:newopt4}) 
can be achieved over the class of $(s, y, \gamma)$-only policies. Thus, there exists a $(s, y, \gamma)$-only policy 
$x^*(t)$, $b^*(t)$, $p^*(t)$ that satisfies: 
\begin{eqnarray}
\expect{b^*(t)p^*(t)a^*(t) - \gamma(t)x^*(t)} = \phi^*  \label{eq:optimalitynew1} \\
\expect{a^*(t) - s(t) - x^*(t)} \leq 0 \label{eq:optimalitynew2} 
\end{eqnarray}
where $a^*(t)$ represents the random requests on slot $t$ associated with pricing decisions $b^*(t)$, $p^*(t)$
and under the random demand state $y(t)$ and the random market price $\gamma(t)$. It is useful to define
$a^* \defequiv \expect{a^*(t)}$.  In the case when the policy $p^*(t)$, $b^*(t)$, $x^*(t)$ that satisfies (\ref{eq:optimalitynew1})-(\ref{eq:optimalitynew2}) is not unique, we define $a^*$ as the maximum value such that there exists
an $(s, y, \gamma)$-only policy that satisfies (\ref{eq:optimalitynew1})-(\ref{eq:optimalitynew2}). 

\subsection{The Joint Pricing and Allocation Algorithm} 

\begin{thm} \label{thm:2} Assume that $x_{max} \geq \max[a_{max}, \epsilon]$,
and that $Q(0) = Z(0) = 0$.  If the above joint pricing and allocation policy is implemented every slot 
with fixed parameters $\epsilon\geq 0$, $V>0$, then: 

a) The worst case delay $D_{max}$ and backlog $Q_{max}$  are 
the same as before (given in (\ref{eq:Dmax2}), (\ref{eq:Qmax})), where $Q_{max}$ is proportional to 
$V$ and $D_{max}$ is 
proportional to $V/\epsilon$. 

b) If the vector $(s(t), y(t), \gamma(t))$ is i.i.d. over slots, and if 
$\epsilon \leq \max[a^*, \expect{s(t)}]$ (where $a^*$ is defined in Section \ref{section:price-optimality}), 
then:\footnote{Note that \emph{actual profit} can be defined
$\tilde{\phi}(t) \defequiv b(t)p(t)a(t) - \gamma(t)\tilde{x}(t)$, with $\tilde{x}(t)$ defined in (\ref{eq:x-tilde}). 
Clearly  $\tilde{\phi}(t) \geq \phi(t)$ for all $t$, and so the time average of the actual profit $\tilde{\phi}(t)$ 
is even closer to
the optimal  value $\phi^*$.} 
\begin{eqnarray*}
& \frac{1}{t}\sum_{\tau=0}^{t-1} \expect{\phi(\tau)} \geq \phi^* - B/V&  \: \: \forall t > 0
\end{eqnarray*} 
where $B$ is defined in (\ref{eq:B}), and $\phi^*$ is the optimal time average profit that can be achieved
by any algorithm that satisfies the constraints of the problem (\ref{eq:newopt1})-(\ref{eq:newopt4}). 
\end{thm} 

\begin{proof} 
See  Appendix C. 
\end{proof} 

\section{Non-I.I.D. Models} \label{section:non-iid}

Here we extend the analysis to treat non-i.i.d. models.  For brevity, we consider only the problem of Section 
\ref{section:model} that seeks to 
allocate $x(t)$  without regard to pricing.\footnote{Similar analysis can be applied to the pricing problem for this non-i.i.d. 
case, using the technique in \cite{neely-inventory-control-arxiv} that incorporates the random demand
$a(t)$ with expectation $F(p(t), y(t), \gamma(t))$, where the $y(t)$ and $\gamma(t)$ processes are arbitrary sample
paths.} 
Specifically, we assume that the processes $s(t)$, $a(t)$, $\gamma(t)$
vary randomly over slots according to any probability 
model (with arbitrary time correlations).   
However, we continue to assume the sample paths are
bounded so that $0 \leq s(t) \leq s_{max}$, $0 \leq a(t) \leq a_{max}$, $0 \leq \gamma(t) \leq \gamma_{max}$ for all $t$. 
We show that the same algorithm of Section \ref{section:model}, which allocates $x(t)$ according to (\ref{eq:x-choice}), 
still provides efficient performance in this context.  We assume that $Q(0) = Z(0) = 0$, and that fixed parameters
$V>0$ and $\epsilon\geq0$ are used. We continue to assume that $x_{max} \geq \max[a_{max}, \epsilon]$. 

We first observe that the exact same worst case backlog and delay bounds $Q_{max}$ and  
$D_{max}$ given in (\ref{eq:Qmax}) and (\ref{eq:Dmax2}) hold in this non-i.i.d. 
case.  Thus,  worst case delay is still bounded by a constant that is proportional to $V/\epsilon$.  This is because the
proof of this bound in Theorem \ref{thm:performance} (a) and (b) was a \emph{sample path proof} that did
not make use of the i.i.d. assumptions.  Indeed, it used only the fact that $0 \leq a(t) \leq a_{max}$ for all $t$. 

It remains only to understand the efficiency of the time average cost.  To this end, we use the $T$-slot lookahead
metric as defined in the universal scheduling work \cite{neely-universal-scheduling}\cite{neely-stock-arxiv}. 
Specifically, suppose that the sample path of $(s(t), a(t), \gamma(t))$ is chosen at time $0$ for all $t$ according to some
arbitrary values.  For a given positive integer $T$ and a positive integer $R$, we consider the first $RT$ slots, 
composed of $R$ successive ``frames'' of size $T$. For each frame $r \in \{0, 1, \ldots, R-1\}$, we define $c_r^*$ as the 
optimum solution to the following ``ideal'' problem that uses full knowledge of $(s(t), a(t), \gamma(t))$ over the frame: 
\begin{eqnarray}
& \mbox{Minimize:}  \: \: \:  c_r^* \defequiv \frac{1}{T}\sum_{\tau=rT}^{(r+1)T-1} \gamma(t)x(t) \label{eq:universal1} \\
& \mbox{Subject to:} \: \: 1)  \sum_{\tau=rT}^{(r+1)T-1} [s(\tau) + x(\tau) - a(\tau)] \geq 0 \label{eq:universal2} \\
& \: \: 2)  \sum_{\tau=rT}^{(r+1)T-1} [s(\tau) + x(\tau) - \epsilon] \geq 0  \label{eq:universal3} \\
& \: \: \: \: \: \: 3) 0 \leq x(\tau) \leq x_{max} 
  \forall \tau \in \{rT, \ldots, (r+1)T-1\} \label{eq:universal4} 
\end{eqnarray}
Thus, $c_r^*$ is the optimal cost that can be achieved over frame $r$, considering all possible ways of allocating
$x(\tau)$ over this frame using perfect knowledge of the future values of $(s(\tau), a(\tau), \gamma(\tau))$ over this
frame, subject to ensuring the total energy provided over the frame is at least as much as the total sum arrivals, 
and  is also at least $\epsilon T$. 

\begin{thm} \label{thm:universal}  (Universal Scheduling) 
Under the above assumptions, the worst case backlog and delay are given by $Q_{max}$ and 
$D_{max}$ in (\ref{eq:Qmax}) and (\ref{eq:Dmax2}). Further, 
for all positive integers $T$ and $R$, we have: 
\begin{eqnarray*} 
&\frac{1}{RT} \sum_{\tau=0}^{RT-1} \gamma(\tau)x(\tau) \leq \frac{1}{R}\sum_{r=0}^{R-1} c_r^* + \frac{BT}{V} &
\end{eqnarray*}
where $B$ is defined in (\ref{eq:B}).
\end{thm} 
\begin{proof} 
The proof combines the techniques of the proof of Theorem \ref{thm:performance} with the universal 
scheduling results in \cite{neely-universal-scheduling}\cite{neely-stock-arxiv}, and is given in Appendix B. 
\end{proof} 

The above result says that the achieved time average cost over any interval of $RT$ slots is less than or equal to the 
average of the $c_r^*$ values, plus a ``fudge factor'' of at most $BT/V$. While the average of the $c_r^*$ 
values is \emph{not} the same as the minimum cost that could be achieved with perfect knowledge of the future
over the full $RT$ slots, this result is still interesting because the $c_r^*$ values are still obtained by ideal algorithms
implemented over $T$ slot frames with full knowledge of the future events in these frames.

\section{Experimental Evaluation}
\label{section:exp}

We evaluated the performance of the proposed algorithm on a six-month data set that we created by combining 
$10$-minute average spot market prices $\gamma(t)$ 
for Los Angeles area (LA1) from CAISO~\cite{oasis}
and $10$-minute energy production $s(t)$ 
for a small subset of windfarms from the Western Wind resources 
Dataset published by the National Renewable Energy Laboratory~\cite{NREL}. 
We modeled the demand $a(t)$ as i.i.d. over slots and uniformly distributed 
over the integers $\{0, 1, \ldots, a_{max}\}$. 
We executed the proposed Lyapunov drift optimization algorithm in $10$-minute timeslots and experimented with different values of the parameters $V,\epsilon$ and the corresponding deadlines they generate.

We compare the proposed algorithm against a simple greedy strategy ``Purchase at deadline,'' 
which tries to use all the available resource $s(t)$ and only buys from the spot market as a last resort if a deadline is reached. 
As can be seen in Fig. \ref{fg_comp1}, the proposed algorithm reduces the cost of the renewable supplier by approximately a factor of $2$ in the tested six-month window. The slope of the two lines is different, suggesting that the savings are unbounded as the time increases. This is not surprising since the greedy strategy does not hedge for future high prices in the spot market while the proposed algorithm learns to proactively buy when the spot market prices are lower than typical and deadline violations seem probable. The high 
variability of the spot market prices~\cite{oasis} makes this advantage significant. 
The second observation, seen in Fig. \ref{fg_comp2}, is that the proposed algorithm has on average a 
much smaller delay than the deadline, which for our parameters was $D_{max}=70$ hours. 
On the contrary, the greedy algorithm makes many requests wait close to (or exactly at) 
the maximum 
allowed $70$ hours.

Our results use $\epsilon = \expect{a(t)} = a_{max}/2$.  We also conducted simulations
with $\epsilon=0$,  which does not require knowledge of $\expect{a(t)}$.  
While $\epsilon =0$ does not provide a finite delay guarantee, 
it still guarantees the same finite $Q_{max}$. Together with FIFO
service, this means that the worst case 
delay for requests that arrive at time $t$ is given by the smallest integer $T>0$ such that 
$\sum_{\tau=t+1}^{t+T} s(\tau) \geq Q_{max}$. While there is no bound on this for general $s(t)$ processes, 
it can still lead to small
delays.  Indeed,  in the simulations 
it \emph{still} maintained all
delays under $D_{max} = 2.9$ days (having a maximum experimental delay of 14 hours, as compared
to 9.5 hours for the $\epsilon = \expect{a(t)}$ case).\footnote{For legibility, the delay data for the $\epsilon=0$ case is not
shown in Fig. \ref{fg_comp2}.} 
Fig. \ref{fg_comp1} shows it gives slightly better cost, particularly because it increases delay.
Both Lyapunov optimization  algorithms provided 
significantly better cost and delay as compared to the greedy algorithm. 
It should be noted that we did not compare against dynamic programming algorithms such
as the one proposed in \cite{pap3}. While it is clear that a dynamic programming approach could solve this problem optimally if the statistics of the underlying processes were known, one benefit of our approach is that no such prior knowledge is required. Further, the Lyapunov approach yields an efficient algorithm for multiple queues corresponding to different customers with different deadlines.

\begin{figure}[t]
\begin{center}
\includegraphics[width=80mm]{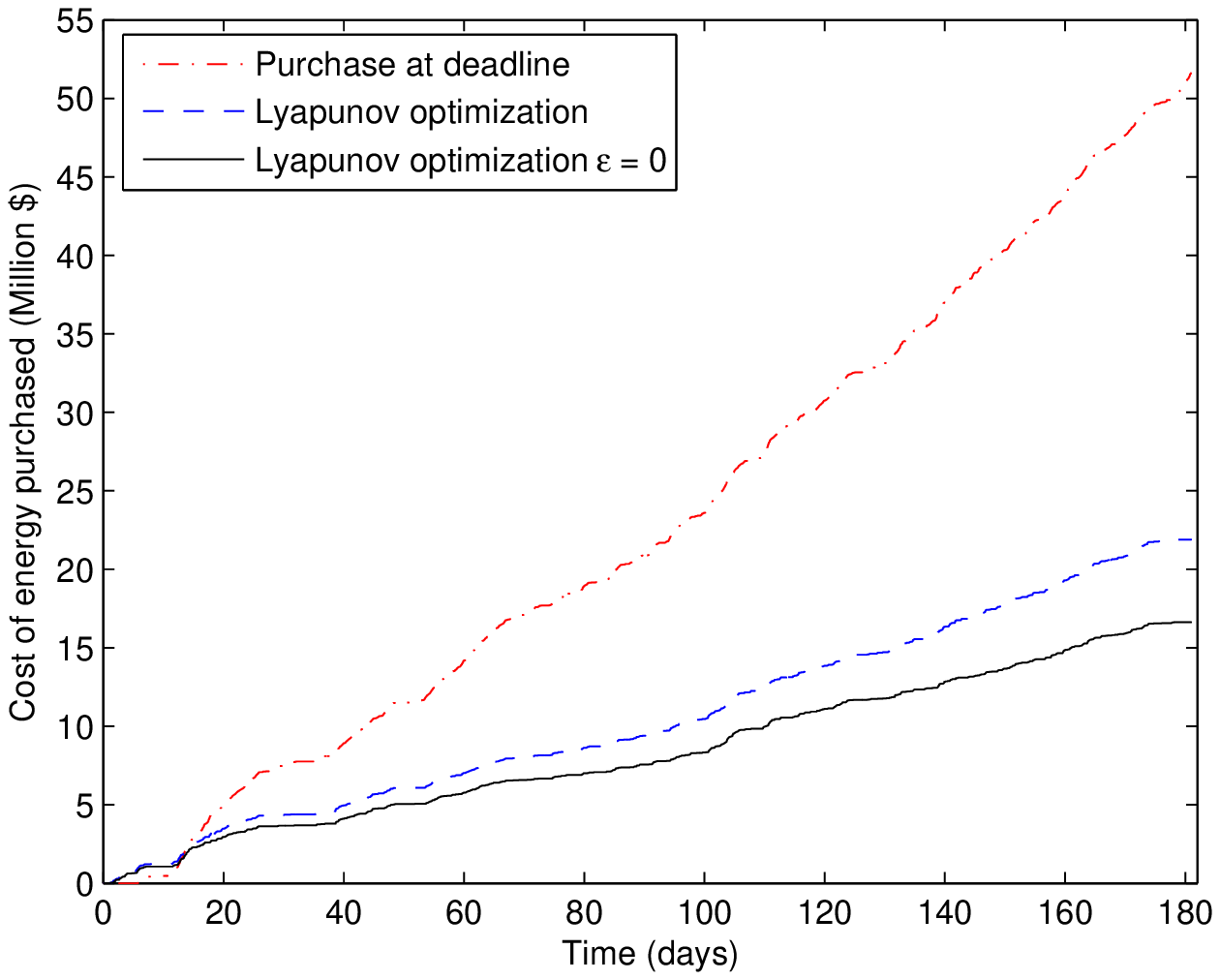}
\vspace{-.5cm}
\caption{Cost of the renewable energy supplier for energy purchased at the spot market. 
For the proposed algorithm we used the parameters $a_{max}=175, \gamma_{max} = 180, 
x_{max}=400, V=100, D_{max}=415=2.9$ days. 
}
\label{fg_comp1}
\includegraphics[width=80mm]{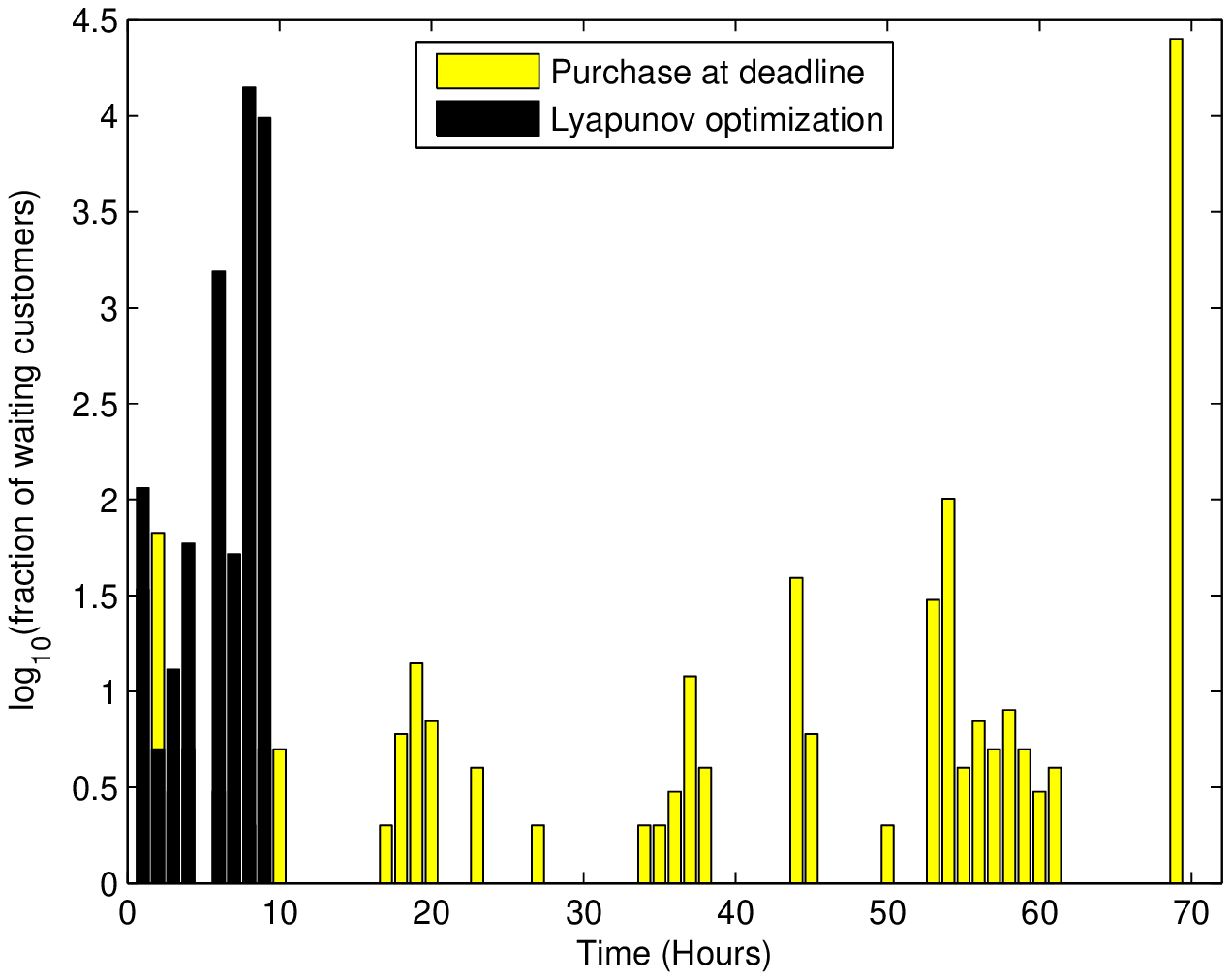}
\vspace{-.5cm}
\caption{Histogram of delay for  the customers waiting in the service queues of the renewable energy supplier under the two algorithms (vertical axis in logarithmic scale). Case $\epsilon=0$ is not shown, but has max delay $14$ hours, 
as compared with the $\epsilon=\expect{A(t)}$ case (shown) with max delay 9.5 hours. $\vspace{-.3in}$}
\label{fg_comp2}
\end{center}
\vspace{-0.2cm}
\end{figure}

We now present some further experimental results investigating the influence of varying $V$ 
and $\epsilon$ in the performance of the proposed algorithm. 
For these simulations we used the same data set as the previous part. For the first experiment, the performance of the algorithm for different values of parameter $V$ is compared. The rest of the parameters are
unchanged and are $a_{max} = 175$, $\gamma_{max} = 180$, 
$x_{max} = 400$, $s_{max} = 90$, and $\epsilon = 87.5$. The result is shown in Fig. \ref{compV}. As expected, 
the cost decreases with $V$.  The tradeoff is in the 
maximum waiting time of the packets.  The maximum waiting times observed in the simulations 
for parameter $V$ being $20,50,100,200$ are $3.5,5.8,10.2,15.2$ hours, respectively.\footnote{The maximum observed
waiting time for the simulation run for the $V=100$ case of Fig. \ref{compV}  
was 10.2, rather than 9.5 as in the previous simulation
for the case $V=100$.  This is because this simulation
used independently generated $a(t)$ values.} 
\begin{figure}[t]
\begin{center}
\includegraphics[width=80mm]{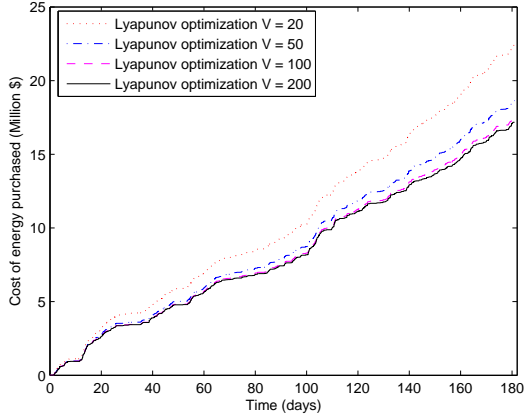}
\vspace{-.5cm}
\caption{Cost of the renewable energy supplier for energy purchased at the spot market for different values of $V = 20,50,100,200$. 
For the proposed algorithm we used the parameters $a_{max}=175$, $x_{max}=400$, and $\epsilon = 87.5$.
}
\label{compV}
\end{center}
\end{figure}
For the second experiment, we consider 
the performance of our algorithm for different values of $\epsilon$. Here, we fixed the value $V = 100$ and run the simulation for $\epsilon = \{87.5,60,35,10\}$. The cost decreases as $\epsilon$ decreases, as shown in Fig. \ref{compEps}. However, the maximum observed waiting times increase with $\epsilon$. 
So for $\epsilon = 87.5,60,35,10$, the maximum observed
waiting times are $9.5,11.7,12.5,13.7$ hours, respectively.
Overall, as expected, the cost gets better as $V$ is increased, with a tradeoff in 
waiting time. Further, the waiting time reduces as $\epsilon$ increases 
to $\expect{A(t)}$, although waiting times are still reasonable even with $\epsilon=0$, which is useful when 
$\expect{A(t)}$ is unknown. For non-i.i.d. situations, 
using a smaller value of $\epsilon$ may also reduce cost due to the fact that this relaxes the constraint 
(\ref{eq:universal3}). 

\begin{figure}[t]
\begin{center}
\includegraphics[width=80mm]{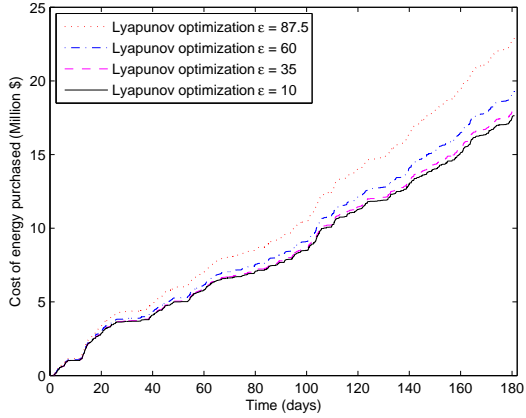}
\vspace{-.5cm}
\caption{Cost of the renewable energy supplier for energy purchased at the spot market for different values of $\epsilon = 87.5,60,35,10$. 
For the proposed algorithm we used the parameters $a_{max}=175, x_{max}=400$, and $V = 100$.
}
\label{compEps}
\end{center}
\end{figure}

\section{Conclusions} 

This work presents a Lyapunov optimization approach to the problem of efficient use of renewable
energy sources.  Efficiency can be improved if consumers are flexible and can tolerate their requests being
served with some delay.   
Two different problems were presented:  One that seeks to minimize cost associated with 
using an outside (possibly non-renewable) plant to meet the deadlines, and another that seeks to maximize
profit by dynamically selecting a price for service.  Our algorithms are simple and were shown to operate 
efficiently without knowing the statistical properties of the supply, demand, and energy request processes. 
We first considered a simple case when these processes are i.i.d. over slots but with unknown
probabilities.  We next treated the 
general case of arbitrary (possibly non-i.i.d. and non-ergodic) sample paths. 
Our analysis also contributes to the theory of Lyapunov optimization by introducing a new type of virtual 
queue that guarantees a bounded
worst case delay.  Our algorithms use a parameter $V$ that can be tuned as desired
to affect a performance-delay tradeoff, where achieved cost is within $O(1/V)$ from optimal, with a worst case
delay guarantee that is $O(V)$.   These techniques provide a convenient alternative to dynamic programming
that leads to a general framework for problems that naturally arise in scheduling of renewable energy markets. 
%

\section*{Appendix A -- Proof of Lemma \ref{lem:drift-bound}} 

From the $Z(t)$ update rule (\ref{eq:z-update}) we have: 
\[ Z(t+1) \leq \max[Z(t) - s(t) - x(t) + \epsilon, 0] \]
and hence: 
\[ Z(t+1)^2 \leq (Z(t) - s(t) - x(t) + \epsilon)^2 \]
Thus: 
\begin{eqnarray*}
\frac{Z(t+1)^2 - Z(t)^2}{2} \leq \\
\frac{1}{2}(\epsilon - s(t) - x(t))^2 + Z(t)(\epsilon - s(t) - x(t)) \\
\leq \frac{1}{2}\max[(s_{max} + x_{max})^2, \epsilon^2] + Z(t)(\epsilon - s(t) -x(t)) 
\end{eqnarray*}
Similarly, by squaring (\ref{eq:q-update}) and using the inequality: 
\[ (\max[Q - \mu, 0] + a)^2 \leq Q^2 + \mu^2 + a^2 + 2Q(a-\mu) \]
which holds for any $Q\geq 0$, $\mu\geq 0$, $a\geq 0$, 
we obtain: 
\begin{eqnarray}
\frac{Q(t+1)^2 - Q(t)^2}{2} \leq \frac{1}{2}[(s_{max} + x_{max})^2 + a_{max}^2] \nonumber \\
+ Q(t)(a(t) - s(t) - x(t)) \label{eq:q-one-slot} 
\end{eqnarray}
Combining the above yields: 
\begin{eqnarray}
&& \hspace{-.5in} L(\bv{\Theta}(t+1)) - L(\bv{\Theta}(t)) \leq B \nonumber \\
&&  \hspace{-.4in}+ Q(t)(a(t) - s(t) - x(t))  
+ Z(t)(\epsilon - s(t) - x(t)) \label{eq:drift-step} 
\end{eqnarray}
Taking conditional expectations of the above, given $\bv{\Theta}(t)$, and adding 
$V\expect{\gamma(t)x(t)|\bv{\Theta}(t)}$ to both sides proves the result.

\section*{Appendix B -- Proof of Theorem \ref{thm:universal}}

 Again define $\bv{\Theta}(t) \defequiv [Q(t), Z(t)]$, and  define the Lyapunov function $L(\bv{\Theta}(t))$ the same
 as before: 
 \[ L(\bv{\Theta}(t)) \defequiv \frac{1}{2}[Q(t)^2 + Z(t)^2] \]
As in \cite{neely-stock-arxiv}\cite{neely-universal-scheduling}, for a given integer $T>0$, 
we define the \emph{$T$-slot sample
path drift} $\Delta_T(\bv{\Theta}(t))$ as follows: 
\[ \Delta_T(\bv{\Theta}(t)) \defequiv L(\bv{\Theta}(t+T)) - L(\bv{\Theta}(t)) \]
This differs from our 1-slot conditional drift $\Delta(\bv{\Theta}(t))$, used for the i.i.d. analysis, because
(i) It involves $T$ slots, rather than 1 slot, and (ii) It does not use an expectation.   

Now suppose that the values $(a(\tau), s(\tau), \gamma(\tau))$ and $x(\tau)$ satisfy the following for all $\tau$:  
\begin{eqnarray}
0 \leq a(\tau) \leq a_{max} \: \: , \: \:  0 \leq s(\tau) \leq s_{max} \label{eq:bounds1} \\
  0 \leq \gamma(\tau) \leq \gamma_{max}  \: \: , \: \: 0 \leq x(\tau) \leq x_{max} \label{eq:bounds2} 
\end{eqnarray}
We have the following 
lemma. 

\begin{lem} \label{lem:T-slot-drift} 
Fix any slot $t$, any queue state $\bv{\Theta}(t) = [Q(t), Z(t)]$, and 
any integer $T>0$.  Consider an arbitrary sample path for 
$a(\tau), s(\tau), \gamma(\tau)$, over the interval $\tau \in \{t, t+1, \ldots, t+T-1\}$, 
assumed only to satisfy (\ref{eq:bounds1})-(\ref{eq:bounds2}). 
Assume that 
the decisions for $x(\tau)$ are given by the algorithm (\ref{eq:x-choice}), with queue
updates for $Q(t)$ and $Z(t)$ given by (\ref{eq:q-update}) and (\ref{eq:z-update}). 
Then: 
\begin{eqnarray*}
\Delta_T(\bv{\Theta}(t)) + V\sum_{\tau=t}^{t+T-1} \gamma(\tau)x(\tau) \leq \\
BT^2 
+ V\sum_{\tau=t}^{t+T-1} \gamma(\tau) x^*(\tau) \\
+ Q(t)\sum_{\tau=t}^{t+T-1} [a(\tau) - s(\tau) - x^*(\tau)] \\
+ Z(t)\sum_{\tau=t}^{t-1}[\epsilon - s(\tau) - x^*(\tau)]
\end{eqnarray*}
where $x^*(\tau)$ are any alternative choices that satisfy $0 \leq x^*(\tau)\leq x_{max}$
for all $\tau \in \{t, \ldots, t+T-1\}$.  The constant $B$ is given in (\ref{eq:B}).
\end{lem} 
\begin{proof} 
From (\ref{eq:drift-step})  we have that for all $\tau$: 
\begin{eqnarray*}
L(\bv{\Theta}(\tau+1)) - L(\bv{\Theta}(\tau)) \leq B + Q(\tau)(a(\tau) - s(\tau) - x(\tau)) \\
+ Z(\tau)(\epsilon - s(\tau) - x(\tau)) 
\end{eqnarray*}
Summing the result over $\tau\in\{t, \ldots, t+T-1\}$ yields: 
\begin{eqnarray*}
\Delta_T(\bv{\Theta}(t))  \leq BT+ \sum_{\tau=t}^{t+T-1} Q(\tau)(a(\tau) - s(\tau) - x(\tau)) \\
+ \sum_{\tau=t}^{t+T-1}Z(\tau)(\epsilon - s(\tau) -x(\tau)) 
\end{eqnarray*}
Adding the penalty term to both sides yields: 
\begin{eqnarray*}
\Delta_T(\bv{\Theta}(t)) + V\sum_{\tau=t}^{t+T-1}\gamma(\tau)x(\tau)  \leq  BT +  V\sum_{\tau=t}^{t+T-1}\gamma(\tau)x(\tau)\\
+  \sum_{\tau=t}^{t+T-1} Q(\tau)(a(\tau) - s(\tau) - x(\tau)) \\
+ \sum_{\tau=t}^{t+T-1}Z(\tau)(\epsilon - s(\tau) -x(\tau)) 
\end{eqnarray*}
We now use the fact that for each slot $\tau$, the value of $x(\tau)$ is chosen to minimize: 
\[ x(\tau)[V\gamma(\tau) - Q(\tau) - Z(\tau)] \]
over all $x(\tau)$ such that $0 \leq x(\tau) \leq x_{max}$.  It follows that: 
\begin{eqnarray*}
\Delta_T(\bv{\Theta}(t)) + V\sum_{\tau=t}^{t+T-1}\gamma(\tau)x(\tau)  \leq  BT +  V\sum_{\tau=t}^{t+T-1}\gamma(\tau)x^*(\tau)\\
 \sum_{\tau=t}^{t+T-1} Q(\tau)(a(\tau) - s(\tau) - x^*(\tau)) \\
+ \sum_{\tau=t}^{t+T-1}Z(\tau)(\epsilon - s(\tau) -x^*(\tau)) 
\end{eqnarray*}
where for all $\tau \in \{t, \ldots, t+T-1\}$, $x^*(\tau)$ is any value that satisfies $0 \leq x^*(\tau) \leq x_{max}$. 
Now note that the maximum changes in the $Q(\tau)$ and $Z(\tau)$ queues on one slot are given by constants
$C_Q$ and $C_Z$, respectively, defined: 
\begin{eqnarray*}
C_Q &\defequiv& \max[s_{max} + x_{max}, a_{max}] \\
C_Z &\defequiv& \max[s_{max} + x_{max}, \epsilon] 
\end{eqnarray*}
Thus: 
\begin{eqnarray*}
|(Q(\tau) - Q(t))(a(\tau) - s(\tau) -x^*(\tau))|  \leq C_Q^2(\tau-t) \\
|(Z(\tau) - Z(t))(\epsilon - s(\tau) - x^*(\tau)| \leq C_Z^2(\tau-t)  
\end{eqnarray*}
We can thus replace the right hand side of the above drift inequality with: 
\begin{eqnarray*}
\Delta_T(\bv{\Theta}(t)) + V\sum_{\tau=t}^{t+T-1}\gamma(\tau)x(\tau)  \leq  \\
BT +  \frac{(C_Q^2 + C_Z^2)T(T-1)}{2} \\
+ V\sum_{\tau=t}^{t+T-1}\gamma(\tau)x^*(\tau) 
+  Q(t) \sum_{\tau=t}^{t+T-1} (a(\tau) - s(\tau) - x^*(\tau)) \\
+ Z(t)\sum_{\tau=t}^{t+T-1}(\epsilon - s(\tau) -x^*(\tau)) 
\end{eqnarray*}
where we have used the fact that $\sum_{\tau=t}^{t+T-1} (\tau - t) = T(T-1)/2$. 
However, it is not difficult to show that:
\[ \frac{(C_Q^2 + C_Z^2)}{2} \leq B \]
and hence: 
\[ BT + \frac{(C_Q^2 + C_Z^2)}{2}T(T-1) \leq BT^2 \]
This proves the result. 
\end{proof} 

Now fix a frame size $T>0$, consider the timeline decomposed into $R$ successive frames
of size $T$, and consider any frame $r \in \{0, 1, \ldots, R-1\}$.  Define $c_r^*$ as the optimum 
cost in the frame-$r$ problem (\ref{eq:universal1})-(\ref{eq:universal4}), and define $x^*(\tau)$
for $\tau \in \{rT, \ldots, rT+T-1\}$ as the optimal decisions for that problem, which achieve 
$c_r^*$ and satisfy the inequality constraints (\ref{eq:universal2})-(\ref{eq:universal4}). 
Then using the drift bound given in Lemma \ref{lem:T-slot-drift} together with the equalities
and inequalities (\ref{eq:universal1})-(\ref{eq:universal4}), we have: 
\begin{eqnarray*}
\Delta_T(\bv{\Theta}(rT)) + V\sum_{\tau=rT}^{rT + T-1} \gamma(\tau)x(\tau) \leq
BT^2  + VTc_r^*
\end{eqnarray*}
Summing the above over $r \in \{0, 1, \ldots, R-1\}$, using the definition of $\Delta_T(\bv{\Theta}(t))$,
and dividing by $RTV$ yields: 
\begin{eqnarray*}
\frac{L(\bv{\Theta}(RT)) - L(\bv{\Theta}(0))}{RTV} + \frac{1}{RT}\sum_{\tau=0}^{RT-1}\gamma(\tau)x(\tau) \leq  \\
 \frac{BT}{V}
+ \frac{1}{R}\sum_{r=0}^{R-1}c_r^*
\end{eqnarray*}
Using the fact that $L(\bv{\Theta}(0)) = 0$ and $L(\bv{\Theta}(RT)) \geq 0$ yields the result.

\section*{Appendix C -- Proof of Theorem \ref{thm:2}} 

Part (a) follows by noting that the proof of parts (a) and (b) in Theorem \ref{thm:performance} hold exactly 
in this new context, as we have not changed the queueing dynamics for $Q(t)$ or $Z(t)$ or  
the fact that $a(t) \leq a_{max}$ for all $t$. 

We now prove part (b).  We have assumed that $\epsilon \leq \max[a^*, \expect{s(t)}]$. 
We first prove the result for the case $\epsilon \leq a^*$. 
On 
each slot $t$ our dynamic algorithm makes actions $b(t)$, $p(t)$, $x(t)$ that, given the 
observed $\bv{\Theta}(t) = [Q(t), Z(t)]$, 
minimizes the right hand side of the drift inequality 
(\ref{eq:price-drift}) over all alternative choices.  Thus:
\begin{eqnarray}
\Delta(\bv{\Theta}(t)) - V\expect{\phi(t)|\bv{\Theta}(t)} \leq B  \nonumber \\
- V\expect{b^*(t)p(t)F(p^*(t), y(t), \gamma(t)) - \gamma(t)x^*(t)|\bv{\Theta}(t)} \nonumber \\
+ Q(t)\expect{b^*(t)F(p^*(t), y(t), \gamma(t)) - s(t) - x^*(t)|\bv{\Theta}(t)} \nonumber \\
+ Z(t)\expect{\epsilon - s(t) - x^*(t)|\bv{\Theta}(t)} \label{eq:price-drift2}  
\end{eqnarray}
where $b^*(t)$, $p^*(t)$, $x^*(t)$ are any other choices that satisfy: 
\[ 0 \leq x^*(t) \leq x_{max} \: , \: 0 \leq p^*(t) \leq p_{max} \: , \: b^*(t) \in \{0,1\} \: \: \forall t \]
We now use the existence of a $(s,y,\gamma)$-only policy $x^*(t)$, $b^*(t)$, $p^*(t)$ that 
satisfies the inequalities (\ref{eq:optimalitynew1})-(\ref{eq:optimalitynew2}).  It is not 
difficult to show that (\ref{eq:optimalitynew1})-(\ref{eq:optimalitynew2}) are equivalent to the following: 
\begin{eqnarray}
\expect{b^*(t)p^*(t)F(p^*(t), y(t), \gamma(t)) - \gamma(t)x^*(t)|\bv{\Theta}(t)} = \phi^* \label{eq:onew1} \\
\expect{b^*(t)F(p^*(t), y(t),\gamma(t)) - s(t) - x^*(t)|\bv{\Theta}(t)} \leq 0 \label{eq:onew2} \\
\expect{b^*(t)F(p^*(t),y(t),\gamma(t))|\bv{\Theta}(t)} = a^* \label{eq:onew3} 
\end{eqnarray}
where the above conditional expectations (\ref{eq:onew1})-(\ref{eq:onew3}) given $\bv{\Theta}(t)$
are the same as the unconditional expectations, because the $(s,y,\gamma)$-only policy does
not depend on the queue states $\bv{\Theta}(t)$ (recall that $(s(t), y(t), \gamma(t))$ is i.i.d. over slots
and hence independent of queue states).  Plugging (\ref{eq:onew1})-(\ref{eq:onew3}) 
directly into the right hand side of (\ref{eq:price-drift2}) yields: 
\begin{eqnarray}
\Delta(\bv{\Theta}(t)) - V\expect{\phi(t)|\bv{\Theta}(t)} \leq B - V\phi^* 
+ Z(t)(\epsilon - a^*) \label{eq:appcfoo1} 
\end{eqnarray}
Because we have assumed that $\epsilon \leq a^*$, this reduces to: 
\begin{eqnarray}
\Delta(\bv{\Theta}(t)) - V\expect{\phi(t)|\bv{\Theta}(t)} \leq B - V\phi^* \label{eq:appcfoo2} 
\end{eqnarray}
Taking expectations of the above (with respect to the random $\bv{\Theta}(t))$ and using
the law of iterated expectations gives: 
\[ \expect{L(\bv{\Theta}(t+1))} - \expect{L(\bv{\Theta}(t))} - V\expect{\phi(t)} \leq B - V\phi^* \]
The above holds for all slots $t$.  Summing over $\tau \in \{0, \ldots, M-1\}$ for some integer $M>0$
yields: 
\begin{eqnarray*}
\expect{L(\bv{\Theta}(M))} - \expect{L(\bv{\Theta}(0))} - V\sum_{\tau=0}^{M-1}\expect{\phi(\tau)} \leq \\
M(B - V\phi^*)
\end{eqnarray*}
Dividing by $VM$ and using the fact that $\expect{L(\bv{\Theta}(0))} = 0$ and $\expect{L(\bv{\Theta}(M))}\geq0$ yields: 
\[ -\frac{1}{M}\sum_{\tau=0}^{M-1} \expect{\phi(\tau)} \leq -\phi^* + B/V \]
This holds for all $M>0$, proving the result for the case $\epsilon \leq a^*$. 

We have used the fact that $\epsilon \leq a^*$ only in showing the $Z(t)(\epsilon - a^*)$ term on the right
hand side of (\ref{eq:appcfoo1}) can be removed while preserving the inequality.  However, suppose that
$\epsilon \leq \expect{s(t)}$.  Then the $Z(t)\expect{\epsilon - s(t) - x^*(t)|\bv{\Theta}(t)}$ term in the
right hand side of (\ref{eq:price-drift2}) can immediately be removed (recall that $x^*(t) \geq 0$ and
$\expect{s(t)} = \expect{s(t)|\bv{\Theta}(t)}$
because $s(t)$ is i.i.d. over slots and hence independent of current queue backlog).  This  leads 
directly to (\ref{eq:appcfoo2}) 
regardless of the value of $a^*$.  Thus, the result holds whenever $\epsilon \leq \max[a^*, \expect{s(t)}]$, proving
the theorem.

\bibliographystyle{unsrt}

\begin{thebibliography}{10}


\bibitem{pap1}
	A. Papavasiliou and S. S. Oren, 
\newblock	Coupling Wind Generation with Deferrable Loads.
\newblock{\em Proceedings of the IEEE Energy 2030 Conference},
	Atlanta, Georgia November 17-18, 2008.

\bibitem{pap2}
A. Papavasiliou, S. S. Oren, M. Junca, A.G. Dimakis, and T. Dickhoff, 
\newblock Coupling Wind Generators with Deferrable Loads.
\newblock{\em CITRIS White paper Technical Report}, 2008. 


\bibitem{pap3} 
A. Papavasiliou and S. S. Oren,
\newblock Supplying Renewable Energy to Deferrable 
Loads: Algorithms and Economic Analysis.
\newblock{\em IEEE PES General Meeting}, Minneapolis, Minnesota, July 25-29 2010.


\bibitem{enerex}
R.~Zavadil, ``2006 {Minessota Wind Integration Study},'' in {\em The Minnesota
  Public Utilities Commission Technical report, Vol I}, Nov 2006.

\bibitem{caiso}
C.~Loutan and D.~Hawkins, ``Integration of renewable resources.{ Transmission
  and operating issues and recommendations for integrating renewable resources
  on the California ISO-controlled Grid},'' in {\em Technical report,
  California Independent System Operator}, 2007.

\bibitem{texas}
R.~Sioshansi and W.~Short, ``Evaluating the impacts of real-time pricing on the
  usage of wind power generation,'' in {\em The Economics of Energy Markets},
  June 2008.


\bibitem{web1}
\url{http://apps1.eere.energy.gov/news/news_detail.cfm}\url{/news_id=12230}.

\bibitem{web2}
\url{http://www.doe.energy.gov/smartgrid.htm}.

\bibitem{FVH05}
F.~van Hulle.
\newblock Large scale integration of wind energy in the european power supply:
  Analysis, recommendations and issues.
\newblock {\em Technical Report, European Wind Energy Association}, 2005.

\bibitem{now}
L.~Georgiadis, M.~J. Neely, and L.~Tassiulas.
\newblock Resource allocation and cross-layer control in wireless networks.
\newblock {\em Foundations and Trends in Networking}, vol. 1, no. 1, pp. 1-149,
  2006.

\bibitem{neely-energy-it}
M.~J. Neely.
\newblock Energy optimal control for time varying wireless networks.
\newblock {\em IEEE Transactions on Information Theory}, vol. 52, no. 7, pp.
  2915-2934, July 2006.

\bibitem{neely-thesis}
M.~J. Neely.
\newblock {\em Dynamic Power Allocation and Routing for Satellite and Wireless
  Networks with Time Varying Channels}.
\newblock PhD thesis, Massachusetts Institute of Technology, LIDS, 2003.

\bibitem{atilla-fairness}
A.~Eryilmaz and R.~Srikant.
\newblock Fair resource allocation in wireless networks using
  queue-length-based scheduling and congestion control.
\newblock {\em Proc. IEEE INFOCOM}, March 2005.

\bibitem{stolyar-greedy}
A.~Stolyar.
\newblock Maximizing queueing network utility subject to stability: Greedy
  primal-dual algorithm.
\newblock {\em Queueing Systems}, vol. 50, pp. 401-457, 2005.

\bibitem{vijay-allerton02}
R.~Agrawal and V.~Subramanian.
\newblock Optimality of certain channel aware scheduling policies.
\newblock {\em Proc. 40th Annual Allerton Conference on Communication ,
  Control, and Computing, Monticello, IL}, Oct. 2002.

\bibitem{prop-fair-down}
H.~Kushner and P.~Whiting.
\newblock Asymptotic properties of proportional-fair sharing algorithms.
\newblock {\em Proc. of 40th Annual Allerton Conf. on Communication, Control,
  and Computing}, 2002.

\bibitem{neely-universal-scheduling}
M.~J. Neely.
\newblock Universal scheduling for networks with arbitrary traffic, channels,
  and mobility.
\newblock {\em ArXiv technical report}, arXiv:1001.0960v1, Jan. 2010.

\bibitem{neely-inventory-control-arxiv}
M.~J. Neely and L.~Huang.
\newblock Dynamic product assembly and inventory control for maximum profit.
\newblock {\em ArXiv Technical Report}, April 2010.

\bibitem{longbo-two-price-ton}
L. Huang and M. J. Neely.
\newblock The Optimality of Two Prices: Maximizing Revenue in a Stochastic
Communication System.
\newblock{\em IEEE/ACM Transactions on Networking}, 
vol. 18, no. 2, pp. 406-419, April 2010.

\bibitem{neely-stock-arxiv}
M. J. Neely. 
\newblock Stock Market Trading via Stochastic Network Optimization.
\newblock ArXiv Technical Report, arXiv:0909.3891v1, Sept. 2009. 

\bibitem{oasis}
California ISO Open Access Same-time Information System (OASIS) 
10-Minute Settlement Interval Average Prices
\url{http://oasishis.caiso.com/}

\bibitem{NREL}
Western Wind resources Dataset, The National Renewable Energy Laboratory
\url{http://wind.nrel.gov/Web_nrel/}







\end{thebibliography}

\end{document}